\documentclass[12pt]{amsart}
\usepackage{amsthm}
\usepackage{amssymb}
\usepackage{longtable}
\usepackage{geometry}
\usepackage{enumerate}
\usepackage{setspace}
\usepackage{upgreek}
\usepackage{tikz}
\usepackage{tikz-cd}
\usetikzlibrary{matrix, arrows}
\usepackage[unicode=true]
 {hyperref}
\hypersetup{
colorlinks=true,
urlcolor=black,
citecolor=blue,
linkcolor=blue,
}

\allowdisplaybreaks

\newtheorem{thm}{Theorem}[section]
\newtheorem{prop}[thm]{Proposition}
\newtheorem{lem}[thm]{Lemma}
\newtheorem{cor}[thm]{Corollary}

\theoremstyle{definition}
\newtheorem{definition}[thm]{Definition}

\theoremstyle{remark}

\numberwithin{equation}{section}

\newcommand{\ff}{\mathfrak{f}}
\newcommand{\fg}{\mathfrak{g}}
\newcommand{\fh}{\mathfrak{h}}
\newcommand{\fpf}{\mathrm{\tiny fpf}}

\newcommand{\bZ}{\mathbb{Z}}
\newcommand{\conj}{\mathrm{conj}}

\newcommand{\Perm}{\mathrm{Perm}}
\newcommand{\Soc}{\mathrm{Soc}}
\newcommand{\Map}{\mathrm{Map}}
\newcommand{\Hol}{\mathrm{Hol}}
\newcommand{\NHol}{\mathrm{NHol}}
\newcommand{\Aut}{\mathrm{Aut}}
\newcommand{\Reg}{\mathrm{Reg}}
\newcommand{\End}{\mathrm{End}}
\newcommand{\Inn}{\mathrm{Inn}}
\newcommand{\Out}{\mathrm{Out}}
\newcommand{\Hom}{\mathrm{Hom}}
\newcommand{\Norm}{\mathrm{Norm}}
\newcommand{\HH}{\mathcal{H}}

\begin{document}

\large 

\title[On the multiple holomorph of a finite almost simple group]{On the multiple holomorph of \\a finite almost simple group}

\author{Cindy (Sin Yi) Tsang}
\address{School of Mathematics (Zhuhai), Sun Yat-Sen University, China}
\email{zengshy26@mail.sysu.edu.cn}\urladdr{http://sites.google.com/site/cindysinyitsang/} 

\date{\today}

\maketitle

\begin{abstract}Let $G$ be a group. Let $\Perm(G)$ denote its symmetric group and write $\Hol(G)$ for the normalizer  of the subgroup of left translations in $\Perm(G)$. The multiple holomorph $\NHol(G)$ of $G$ is in turn defined to be the normalizer of $\Hol(G)$ in $\Perm(G)$. In this paper, we shall show that the quotient group $\NHol(G)/\Hol(G)$ has order two when $G$ is finite and almost simple. As an application of our techniques, we shall also develop a method to count the number of Hopf-Galois structures of isomorphic type on a finite almost simple extension in terms of fixed point free endomorphisms.
\end{abstract}

\tableofcontents

\section{Introduction}

Let $G$ be a group and write $\Perm(G)$ for its symmetric group. Recall that a subgroup $N$ of $\Perm(G)$ is said to be \emph{regular} if the map
\[ \xi_{N} : N\longrightarrow G;\hspace{1em} \xi_{N}(\eta) = \eta(1_G)\]
is bijective, or equivalently, if the $N$-action on $G$ is both transitive and free. For example, both $\lambda(G)$ and $\rho(G)$ are regular subgroups of $\Perm(G)$, where
\[\begin{cases}
\lambda:G\longrightarrow\Perm(G);\hspace{1em}\lambda(\sigma) = (x\mapsto \sigma x)\\
\rho:G\longrightarrow\Perm(G);\hspace{1em}\rho(\sigma) = (x\mapsto x\sigma^{-1})
\end{cases}\]
denote the left and right regular representations of $G$, respectively. Plainly, we have $\lambda(G)$ are $\rho(G)$ are equal precisely when $G$ is abelian. Recall further that the \emph{holomorph of $G$} is defined to be
\begin{equation}\label{Hol(G)} \Hol(G) = \rho(G) \rtimes \Aut(G).\end{equation}
Alternatively, it is not hard to verify that
\[\Norm_{\Perm(G)}(\lambda(G)) = \Hol(G) = \Norm_{\Perm(G)}(\rho(G)).\]
Then, it seems natural to ask whether $\Perm(G)$ has other regular subgroups which also have normalizer equal to $\Hol(G)$. Given any regular subgroup $N$ of $\Perm(G)$, observe that the bijection $\xi_N$ induces an isomorphism
\begin{equation}\label{Xi def}\Xi_N : \Perm(N) \longrightarrow \Perm(G);\hspace{1em}\Xi_N(\pi) = \xi_N\circ\pi\circ\xi_N^{-1}\end{equation}
under which $\lambda(N)$ is sent to $N$. Thus, in turn $\Xi_N$ induces an isomorphism
\[\Hol(N)\simeq \Norm_{\Perm(G)}(N),\]
and so we have
\[ \Norm_{\Perm(G)}(N) = \Hol(G) \mbox{ implies } \Hol(N) \simeq \Hol(G).\]
However, in general, the converse is false, and non-isomorphic groups (of the same order) can have isomorphic holomorphs. Nevertheless, let us restrict to the regular subgroups $N$ which are isomorphic to $G$, and consider
\[\HH_0(G) = \left\lbrace\begin{array}{c}\mbox{regular subgroups $N$ of $\Perm(G)$ isomorphic to $G$}\\\mbox{and such that $\Norm_{\Perm(G)}(N) = \Hol(G)$}\end{array}\right\rbrace.\]
This set was first studied by G. A. Miller \cite{Miller}. More specifically, he defined the \emph{multiple holomorph of $G$} to be
\[ \NHol(G) = \Norm_{\Perm(G)}(\Hol(G)),\]
which clearly acts on $\HH_0(G)$ via conjugation, and he showed that this action is transitive so the quotient group
\[ T(G) = \frac{\NHol(G)}{\Hol(G)}\]
acts regularly on $\mathcal{H}_0(G)$; or see Section~\ref{prelim sec} below for a proof. In \cite{Miller}, he also determined the structure of $T(G)$ for finite abelian groups $G$. Later in \cite{Mills}, W. H. Mills extended this to all finitely generated abelian groups $G$, which was also redone in \cite{Caranti1} using a different approach. Initially, the study of $T(G)$ did not attract much attention, other than \cite{Miller} and \cite{Mills}.  But recently in \cite{Kohl NHol}, T. Kohl revitalized this line of research by computing $T(G)$ for dihedral and dicyclic groups. His work in turn motivated the calculation of $T(G)$ for some other families of finite groups; see \cite{Caranti2} and \cite{Caranti3}. The main purpose of this paper is to compute $T(G)$ for a new family of finite groups $G$.

\vspace{1.5mm}

To explain our motivation, first notice that elements of $\HH_0(G)$ are normal subgroups of $\Hol(G)$; this is known and also see Section~\ref{prelim sec} below for a proof. Instead of $\HH_0(G)$, let us consider the possibly larger sets
\begin{align*}
\HH_1(G) & = \{\mbox{normal and regular subgroups of }\Hol(G)\},\\
\HH_2(G) & = \{\mbox{regular subgroups of $\Hol(G)$ isomorphic to $G$}\}.
\end{align*}
Then, we have the inclusions
\[ \HH_0(G) \subset \HH_1(G)\mbox{ and }\HH_0(G)\subset\HH_2(G).\]
If $G$ is finite and non-abelian simple, then we know that
\[ \HH_2(G) = \{\lambda(G),\rho(G)\}\]
by the proof of \cite[Theorem 4]{Childs non-abelian}, and this in turn implies that
\begin{equation}\label{T order 2} \HH_0(G) = \{\lambda(G),\rho(G)\}\mbox{ whence }T(G) \simeq \bZ/2\bZ.\end{equation}
Inspired by this observation, it seems natural to ask whether the same or at least a similar phenomenon holds for other finite groups $G$ which are close to being non-abelian simple. More precisely, let us consider the following three generalizations of non-abelian simple groups.

\begin{definition}A group $G$ is said to be
\begin{enumerate}[(1)]
\item \emph{quasisimple} if $G=[G,G]$ and $G/Z(G)$ is simple, where $[G,G]$ and $Z(G)$ denote the commutator subgroup and the center of $G$, respectively.
\item \emph{characteristically simple} if it has no non-trivial proper characteristic subgroup; let us note that for finite $G$, this is equivalent to $G=T^n$ for some simple group $T$ and natural number $n$.
\item \emph{almost simple} if $\Inn(T)\leq G\leq\Aut(T)$ for some non-abelian simple group $T$, where $\Inn(T)$ denotes the inner automorphism group of $T$; let us note that $\Inn(T)$ is the socle of $G$ in this case, as shown in Lemma~\ref{socle lem} below, for example.
\end{enumerate}
\end{definition}

If $G$ is finite and quasisimple, then we know that  
\[  \HH_2(G) = \{\lambda(G),\rho(G)\}\]
by \cite[(1.1) and Theorem 1.3]{Tsang HG}, whence (\ref{T order 2}) holds as above. However, if $G$ is finite and non-abelian characteristically simple or almost simple, then the size of $\HH_2(G)$ can be arbitrarily large as the order of $G$ increases, by \cite{Tsang char simple} and \cite[Theorem 5]{Childs non-abelian}. Nevertheless, if $G$ is finite and non-abelian characteristically simple, then we know that
\[ \HH_1(G) = \{\lambda(G),\rho(G)\}\]
by a special case of \cite[Theorem 7.7]{Caranti2}, and thus (\ref{T order 2}) holds as well. Our result is that if $G$ is finite and almost simple, then the same phenomenon occurs. More specifically, we shall prove:

\begin{thm}\label{main thm}Let $G$ be a finite almost simple group. Then, we have 
\[\HH_1(G) = \{\lambda(G),\rho(G)\}.\]
In particular, the statement (\ref{T order 2}) holds.
\end{thm}


In order to compute $\HH_1(G)$, we shall develop a way to describe the regular subgroups of $\Hol(G)$, and not just the ones which are normal; see Section~\ref{reg sec}. These regular subgroups are closely related to Hopf-Galois structures on field extensions. In particular, by work of \cite{GP} and \cite{By96}, for any finite Galois extension $L/K$ with Galois group $G$, there is a one-to-one correspondence between Hopf-Galois structures on $L/K$ of \emph{type} $G$ and elements of $\HH_2(G)$. We shall refer the reader to \cite[Chapter 2]{Childs book} for more details. Let us mention in passing that there is also a connection between regular subgroups of $\Hol(G)$ and the non-degenerate set-theoretic solutions of the Yang-Baxter equation; see \cite{Skew braces}.

\vspace{1.5mm}

Therefore, other than $\HH_1(G)$, it is also of interest to compute $\HH_2(G)$. If $G$ is finite and non-abelian characteristically simple, then this was already done in \cite{Tsang char simple}. However, if $G$ is finite and almost simple, then as far as the author is aware, the only result is \cite[Theorem 5]{Childs non-abelian}, which says that
\[\label{Sn count} \#\HH_2(S_n) = 2\cdot(1+\#\{\sigma\in A_n : \sigma\mbox{ has order two}\})\mbox{ for all }n\geq 5.\]
Here $S_n$ and $A_n$, respectively, denote the symmetric and alternating groups on $n$ letters. Using the techniques to be developed in Section~\ref{reg sec}, which were largely motivated by \cite{Childs non-abelian}, we shall also generalize the above result as follows. Recall that an endomorphism $f$ on $G$ is said to be \emph{fixed point free} if 
\[ f(\sigma) = \sigma\mbox{ holds precisely when }\sigma = 1_G.\]
Let $\End_{\fpf}(G)$ denote the set of all such endomorphisms. Also, write $\Inn(G)$ for the inner automorphism group $G$. Then, we shall prove:

\begin{thm}\label{main thm'}Let $G$ be a finite almost simple group such that $\Inn(G)$ is the only subgroup of $\Aut(G)$ isomorphic to $G$. Then, we have
\[\#\HH_2(G) = 2\cdot\#\End_\fpf(G).\]
Moreover, in the case that $\Soc(G)$ has prime index $p$ in $G$, we have
\begin{align*} \#\End_\fpf(G) & = 1 + \#\{\sigma\in\Soc(G): \sigma\mbox{ has order $p$}\}\\&\hspace{0.9cm} + (p-2)/(p-1)\cdot\#\{\sigma\in G\setminus\Soc(G) : \sigma\mbox{ has order $p$}\},\end{align*}
where $\Soc(G)$ denotes the socle of $G$.
\end{thm}

It is well-known, or by Lemma~\ref{Aut(G) lem} below, that for $G = \Aut(T)$ where $T$ is a non-abelian simple group, we have $\Aut(G)\simeq G$ and so the first hypothesis of Theorem~\ref{main thm'} is trivially satisfied. Now, among the $26$ sporadic simple groups $T$, exactly $12$ of them have non-trivial outer automorphism group $\Out(T)$, in which case $\Out(T)$ has order two. By plugging in $p=2$ in Theorem~\ref{main thm'}, and using the element structure of $T$ given in the \textsc{Atlas} \cite{ATLAS}, we then obtain the values of $\#\HH_2(G)$ as follows. The notation is the same as in the \textsc{Atlas} \cite{ATLAS}. 
\begin{longtable}{|c|c|c|}
\hline
$T$ & no. of elements of order two in $T$ & $\#\HH_2(G)$ for $G=\Aut(T)$\\\hline\hline
$\mbox{M}_{12}$ & $891$ & $1,784$\\
$\mbox{M}_{22}$ & $1,155$ & $2,312$\\
$\mbox{HS}$ & $21,175$ & $42,352$\\
$\mbox{J}_2$ & $2,835$ & $5,672$\\
$\mbox{McL}$ & $22,275$ & $44,552$\\
$\mbox{Suz}$ & $2,915,055$& $5,830,112$\\
$\mbox{He}$ & $212,415$ & $424,832$\\
$\mbox{HN}$ & $75,603,375$ & $151,206,752$\\
$\mbox{Fi}_{22}$ & $37,706,175$ & $75,412,352$\\
$\mbox{Fi'}_{24}$ & $7,824,165,773,823$ & $15,648,331,547,648$ \\
$\mbox{O'N}$ & $2,857,239$ & $5,714,480$\\
$\mbox{J}_3$ & $26,163$ &$52,328$\\
\hline
\end{longtable}

\vspace{-2mm}

Since $\#\HH_2(T) = 2$ for all finite non-abelian simple groups $T$, the number $\#\HH_2(G)$ is now known for all almost simple groups $G$ of sporadic type.

\vspace{1.5mm}

Finally, let us remark that if $G/\Soc(G)$ is not cyclic (of prime order), then the enumeration of $\End_{\fpf}(G)$ becomes much more complicated. Currently, the author does not have a systematic way of treating the general case.

\section{Preliminaries on the multiple holomorph}\label{prelim sec}

In this section, we shall give a proof of the fact that the action of $\NHol(G)$ on the set $\HH_0(G)$ via conjugation is transitive, and the fact that elements of $\HH_0(G)$ are normal subgroups of $\Hol(G)$. Both of them are already known in the literature and are consequences of the next simple observation.

\begin{lem}\label{reg sub conj}Isomorphic regular subgroups of $\Perm(G)$ are conjugates.
\end{lem}
\begin{proof}Let $N_1$ and $N_2$ be any two isomorphic regular subgroups of $\Perm(G)$. Let $\varphi:N_1\longrightarrow N_2$ be an isomorphism and note that the isomorphism
\[ \Xi_\varphi : \Perm(N_1) \longrightarrow \Perm(N_2);\hspace{1em}\Xi_\varphi(\pi) = \varphi\circ\pi\circ\varphi^{-1}\]
sends $\lambda(N_1)$ to $\lambda(N_2)$. For $i=1,2$, recall that the isomorphism $\Xi_{N_i}$ defined as in (\ref{Xi def}) sends $\lambda(N_i)$ to $N_i$. It follows that $\Xi_{N_2}\circ\Xi_\varphi\circ\Xi_{N_1}^{-1}$ maps $N_1$ to $N_2$. We then deduce that $N_1$ and $N_2$ are conjugates via $\xi_{N_2}\circ\varphi\circ\xi_{N_1}^{-1}$.
\end{proof}

Lemma~\ref{reg sub conj} implies that regular subgroups of $\Perm(G)$ isomorphic to $G$ are precisely the conjugates of $\lambda(G)$. For any $\pi\in\Perm(G)$, we have
\[ \Norm_{\Perm(G)}(\pi \lambda(G)\pi^{-1}) = \pi\Hol(G)\pi^{-1},\]
which is equal to $\Hol(G)$ if and only if $\pi\in\NHol(G)$. It follows that
\begin{equation}\label{H0 new}\HH_0(G) = \{\pi\lambda(G)\pi^{-1} : \pi \in\NHol(G) \}.\end{equation}
Thus, clearly $\NHol(G)$ acts transitively on $\HH_0(G)$ via conjugation. Since the stabilizer of any element of $\HH_0(G)$ under this action is equal to $\Hol(G)$, the quotient $T(G)$ acts regularly on $\HH_0(G)$. For any $\pi\in\Perm(G)$, we have
\[ \pi\lambda(G)\pi^{-1}\lhd \Hol(G)\iff\ \begin{cases}
\pi\lambda(G)\pi^{-1}\leq \Hol(G),\\\Hol(G)\leq\Norm_{\Perm(G)}(\pi\lambda(G)\pi^{-1}).
\end{cases}\]
Since $\lambda(G)\leq\Hol(G)$, both of the conditions on the right are clearly satisfied for $\pi\in\NHol(G)$, and so elements of $\HH_0(G)$ are normal subgroups of $\Hol(G)$. In the case that $G$ is finite, the second condition on the right is satisfied only for $\pi\in\NHol(G)$, whence we have the alternative description
\[ \HH_0(G) = \{ N \lhd \Hol(G):N\simeq G\mbox{ and $N$ is regular}\}\]
for $\HH_0(G)$ in addition to (\ref{H0 new}), and in particular $\HH_0(G) = \HH_1(G)\cap\HH_2(G)$.

\section{Descriptions of regular subgroups in the holomorph}\label{reg sec}


In this section, let $\Gamma$ be a group of the same cardinality as $G$. Plainly, the regular subgroups of $\Hol(G)$ isomorphic to $\Gamma$ are the images of the homomorphisms in the set
\[ \Reg(\Gamma,\Hol(G)) = \{\mbox{injective }\beta\in\Hom(\Gamma,\Hol(G)):\beta(\Gamma) \mbox{ is regular}\}.\]
Note that for $G$ and $\Gamma$ finite, the map $\beta$ is automatically injective when $\beta(\Gamma)$ is regular. Below, we shall give two different ways of describing this set, and it shall be helpful to recall the definition of $\Hol(G)$ given in (\ref{Hol(G)}).

\vspace{1.5mm}

The first description uses bijective crossed homomorphisms.


\begin{definition}\label{xhom def}Given $\ff\in\Hom(\Gamma,\Aut(G))$, a map $\fg\in\Map(\Gamma,G)$ is said to be a \emph{crossed homomorphism with respect to $\ff$} if
\[ \fg(\gamma\delta) = \fg(\gamma)\cdot\ff(\gamma)(\fg(\delta))\mbox{ for all }\gamma,\delta\in \Gamma.\]
Write $Z_\ff^1(\Gamma,G)$ for the set of all such maps. Also, let $Z_\ff^1(\Gamma,G)^*$ and $Z_\ff^1(\Gamma,G)^\circ$, respectively, denote the subsets consisting of those maps which are bijective and injective. Note that these two subsets coincide when $G$ and $\Gamma$ are finite.
\end{definition}

\begin{prop}\label{description1}For $\ff\in\Map(\Gamma,\Aut(G))$ and $\fg\in \Map(\Gamma,G)$, define
\[ \beta_{(\ff,\fg)}: \Gamma\longrightarrow \Hol(G); \hspace{1em}\beta_{(\ff,\fg)}(\gamma) = \rho(\fg(\gamma))\cdot\ff(\gamma).\]
Then, we have
\begin{align*}
\Map(\Gamma,\Hol(G)) & = \{\beta_{(\ff,\fg)}:\ff\in\Map(\Gamma,\Aut(G)),\,\fg\in\Map(\Gamma,G)\},\\
\Hom(\Gamma,\Hol(G)) & = \{\beta_{(\ff,\fg)}: \ff\in\Hom(\Gamma,\Aut(G)),\,\fg\in Z_\ff^1(\Gamma,G) \},\\
\Reg(\Gamma,\Hol(G)) & = \{\mathrm{injective}\mbox{ }\beta_{(\ff,\fg)}: \ff\in\Hom(\Gamma,\Aut(G))\mbox,\,\fg\in Z_\ff^1(\Gamma,G)^*\}. 
\end{align*}
\end{prop}
\begin{proof}This follows easily from (\ref{Hol(G)}); see \cite[Proposition 2.1]{Tsang HG} for a proof and note that the argument there does not require $G$ and $\Gamma$ to be finite.
\end{proof}

The second description uses fixed point free pairs of homomorphisms. The use of such pairs already appeared in \cite{Byott Childs,Childs non-abelian,Childs PAMS} and our Proposition~\ref{description2} below is largely motivated by the arguments on \cite[pp. 83--84]{Childs non-abelian}.

\begin{definition}\label{fpf def}For any groups $\Gamma_1$ and $\Gamma_2$, a pair $(f,h)$ of homomorphisms from $\Gamma_1$ to $\Gamma_2$ is said to be \emph{fixed point free} if the equality $f(\gamma) = h(\gamma)$ holds precisely when $\gamma = 1_{\Gamma_1}$.
\end{definition}

Let $\Out(G)$ denote the outer automorphism group of $G$ and write
\[ \uppi_G : \Aut(G)\longrightarrow\Out(G);\hspace{1em}\uppi_G(\varphi) = \varphi\cdot\Inn(G) \]
for the natural quotient map. Given $\ff\in\Hom(\Gamma,\Aut(G))$, define
\begin{align}\label{Homf}
\Hom_\ff(\Gamma,\Aut(G)) & = \{\fh\in\Hom(\Gamma,\Aut(G)): \uppi_G\circ \ff = \uppi_G\circ \fh\},\\\notag
\Hom_\ff(\Gamma,\Aut(G))^\circ & = \{\fh\in\Hom_\ff(\Gamma,\Aut(G)):(\ff,\fh)\mbox{ is fixed point free}\}.
\end{align}
In view of Proposition~\ref{description1}, it is enough to consider $Z_\ff^1(\Gamma,G)^*$, which in turn is equal to $Z_\ff^1(\Gamma,G)^\circ$ when $G$ and $\Gamma$ are finite. The next proposition gives an alternative description of this latter set in the case that $G$ has trivial center. Let us remark that it may be regarded as a generalization of \cite[Proposition 2.4]{Tsang HG} and is only interesting when $G$ is non-abelian.

\begin{prop}\label{description2}Let $\ff\in\Hom(\Gamma,\Aut(G))$. For $\fg\in Z_\ff^1(\Gamma,G)$, define
\[ \fh_{(\ff,\fg)} : \Gamma\longrightarrow\Aut(G);\hspace{1em}\fh_{(\ff,\fg)}(\gamma) =  \conj(\fg(\gamma))\cdot\ff(\gamma),\]
where $\conj(-) = \lambda(-)\rho(-)$. Then, the map $\fh_{(\ff,\fg)}$ is always a homomorphism. Moreover, in the case that $G$ has trivial center, the maps
\begin{align}\label{h map1}
Z_\ff^1(\Gamma,G)\longrightarrow\Hom_\ff(\Gamma,\Aut(G));\hspace{1em}\fg\mapsto \fh_{(\ff,\fg)}
\\\label{h map2}
Z_\ff^1(\Gamma,G)^\circ \longrightarrow \Hom_\ff(\Gamma,\Aut(G))^\circ;\hspace{1em}\fg\mapsto \fh_{(\ff,\fg)}
\end{align}
are well-defined bijections.
\end{prop}
\begin{proof}First, let $\fg\in Z_\ff^1(\Gamma,G)$. For any $\gamma,\delta\in \Gamma$, we have
\begin{align*}
\fh_{(\ff,\fg)}(\gamma\delta) & = \conj(\fg(\gamma\delta))\cdot\ff(\gamma\delta)\\
& = \conj(\fg(\gamma))\ff(\gamma)\cdot \ff(\gamma)^{-1}\conj(\ff(\gamma)(\fg(\delta))) \ff(\gamma)\cdot\ff(\delta)\\
&=\conj(\fg(\gamma))\ff(\gamma)\cdot\conj(\fg(\delta))\ff(\delta)\\
&=\fh_{(\ff,\fg)}(\gamma)\fh_{(\ff,\fg)}(\delta)
\end{align*}
and so $\fh_{(\ff,\fg)}$ is a homomorphism. This means that (\ref{h map1}) is well-defined.

\vspace{1.5mm}

Now, suppose that $G$ has trivial center, in which case $\conj : G \longrightarrow \Inn(G)$ is an isomorphism. Given $\fh\in\Hom_\ff(\Gamma,\Aut(G))$, define
\[ \fg:\Gamma\longrightarrow G;\hspace{1em} \fg(\gamma) = \conj^{-1}(\fh(\gamma)\ff(\gamma)^{-1}),\]
where $\fh(\gamma)\ff(\gamma)^{-1}\in\Inn(G)$ since $\uppi_G\circ \fh = \uppi_G\circ \ff$. For any $\gamma,\delta\in\Gamma$, we have
\begin{align*}
\conj(\fg(\gamma\delta)) & = \fh(\gamma\delta)\ff(\gamma\delta)^{-1}\\
& = \fh(\gamma)\ff(\gamma)^{-1}\cdot \ff(\gamma)\fh(\delta)\ff(\delta)^{-1}\ff(\gamma)^{-1}\\
& = \conj(\fg(\gamma))\cdot\conj(\ff(\gamma)(\fg(\delta)))\\
& = \conj(\fg(\gamma)\cdot \ff(\gamma)(\fg(\delta)))
\end{align*}
and so $\fg$ is a crossed homomorphism with respect to $\ff$. Clearly $\fh = \fh_{(\ff,\fg)}$ and so this shows that (\ref{h map1}) is surjective. Let $\fg_1,\fg_2,\fg\in Z_\ff^1(\Gamma,G)$. For any $\gamma\in\Gamma$, since $\conj$ is an isomorphism, we have
\[\fg_1(\gamma) = \fg_2(\gamma)\iff \conj(\fg_1(\gamma)) = \conj(\fg_2(\gamma)) \iff \fh_{(\ff,\fg_1)}(\gamma) = \fh_{(\ff,\fg_2)}(\gamma) \]
and so (\ref{h map1}) is also injective. For any $\gamma_1,\gamma_2\in\Gamma$, similarly
\[ \fg(\gamma_1) = \fg(\gamma_2) \iff \conj(\fg(\gamma_1)) = \conj(\fg(\gamma_2)) \iff \fh_{(\ff,\fg)}(\gamma_1^{-1}\gamma_2) = \ff(\gamma_1^{-1}\gamma_2)\]
and this implies that (\ref{h map2}) is a well-defined bijection as well.
\end{proof}

In the case that $G$ is finite, observe that
\[ \#\HH_2(G) = \frac{1}{|\Aut(G)|}\cdot\#\Reg(G,\Hol(G)).\]
From Proposition~\ref{description1}, we then deduce that
\[ \#\HH_2(G) = \frac{1}{|\Aut(G)|}\cdot\#\{(\ff,\fg):\ff\in\Hom(G,\Aut(G)),\,\fg\in Z_\ff^1(G,G)^*\}.\]
By Proposition~\ref{description2}, when $G$ has trivial center, we further have
\begin{equation}\label{HH2 count}
\#\HH_2(G) = \frac{1}{|\Aut(G)|}\cdot\#\left\{(\ff,\fh):\begin{array}{c}
\ff\in\Hom(G,\Aut(G))\\ \fh\in\Hom_\ff(G,\Aut(G))^\circ\end{array}\right\}.
\end{equation}
This formula shall be useful for the proof of Theorem~\ref{main thm'}.

\vspace{1.5mm}

Finally, we shall give a necessary condition for a subgroup of $\Hol(G)$ to be normal in terms of the notation of Propositions~\ref{description1} and~\ref{description2}.

\begin{prop}\label{normal prop}Let $\ff\in\Hom(\Gamma,\Aut(G))$ and $\fg\in Z^1_\ff(\Gamma,G)$. If $\beta_{(\ff,\fg)}(\Gamma)$ is a normal subgroup of $\Hol(G)$, then both $\ff(\Gamma)$ and $\fh_{(\ff,\fg)}(\Gamma)$ are normal subgroups of $\Aut(G)$.
\end{prop}
\begin{proof}Suppose that $\beta_{(\ff,\fg)}(\Gamma)$ is normal in $\Hol(G)$ and let $\varphi\in\Aut(G)$. Then, for any $\gamma\in\Gamma$, there exists $\gamma_\varphi\in\Gamma$ such that
\[ \varphi\beta_{(\ff,\fg)}(\gamma)\varphi^{-1} = \beta_{(\ff,\fg)}(\gamma_\varphi).\]
Rewriting this equation yields
\[ \rho(\varphi(\fg(\gamma)))\cdot\varphi \ff(\gamma)\varphi^{-1} = \rho(\fg(\gamma_\varphi))\cdot \ff(\gamma_\varphi).\]
Since (\ref{Hol(G)}) is a semi-direct product, this in turn gives
\[ \varphi(\fg(\gamma)) = \fg(\gamma_\varphi)\mbox{ and }\varphi\ff(\gamma)\varphi^{-1} = \ff(\gamma_\varphi).\]
The latter shows that $\ff(\Gamma)$ is normal in $\Aut(G)$. The above also implies that
\[ \varphi \fh_{(\ff,\fg)}(\gamma)\varphi^{-1} = \conj(\varphi(\fg(\gamma)))\cdot\varphi\ff(\gamma)\varphi^{-1} = \conj(\fg(\gamma_\varphi))\cdot\ff(\gamma_\varphi) = \fh_{(\ff,\fg)}(\gamma_\varphi),\]
and hence $\fh_{(\ff,\fg)}(\Gamma)$ is normal in $\Aut(G)$ as well.
\end{proof}

\section{Basic properties of almost simple groups}

In this section, let $S$ be an almost simple group and let $T$ be a non-abelian simple group such that $\Inn(T)\leq S \leq \Aut(T)$. Notice that $\Inn(T)$ is normal in $\Aut(T)$ and hence is normal in $S$ as well. Recall the known fact, which is easily verified, that for any $\varphi\in\Aut(T)$, we have
\begin{equation}\label{centralizer}
\varphi\circ\psi = \psi \circ\varphi \mbox{ for all }\psi\in\Inn(T)\mbox{ implies }\varphi = \mbox{Id}_T.
\end{equation}
This implies the next three basic properties of $S$ which we shall need.

\begin{lem}\label{center lem}The center of $S$ is trivial.
\end{lem}
\begin{proof}This follows directly from (\ref{centralizer}). \end{proof}

\begin{lem}\label{socle lem}Every non-trivial normal subgroup of $S$ contains $\Inn(T)$.
\end{lem}
\begin{proof}Let $R$ be any normal subgroup of $S$ such that $R\not\supset\Inn(T)$, or equivalently $R\cap\Inn(T)\neq\Inn(T)$. Then, since $R\cap\Inn(T)$ is normal in $\Inn(T)$ and $\Inn(T)\simeq T$ is simple, this means that $R\cap\Inn(T)= 1$. For any $\varphi\in R$, since both $R$ and $\Inn(T)$ are normal in $\Aut(T)$, we have
\[ \psi\circ\varphi\circ\psi^{-1}\circ\varphi^{-1} \in R\cap\Inn(T)\mbox{ for all }\psi\in\Inn(T).\]
We then deduce from (\ref{centralizer}) that $\varphi = \mbox{Id}_T$ and so $R$ is trivial.
\end{proof}

\begin{lem}\label{Aut(G) lem}There is an injective homomorphism 
\[\iota:\Aut(S)\longrightarrow\Aut(T)\]
such that the composition
\[\begin{tikzcd}[column sep =2cm]
S\arrow{r} & \Inn(S) \arrow{r}{\mbox{\tiny inclusion}} & \Aut(S) \arrow{r}{\iota} &\Aut(T)
\end{tikzcd}\]
is the inclusion map, where the first arrow is the map $\varphi\mapsto(x\mapsto\varphi x\varphi^{-1})$.
\end{lem}
\begin{proof}Put $T^\#= \Inn(T)$, which is the socle of $S$ by Lemma~\ref{socle lem}, and hence is a characteristic subgroup of $S$. Thus, there is a well-defined homomorphism
\begin{equation}\label{embedding} \Aut(S) \longrightarrow \Aut(T^\#);\hspace{1em}\theta\mapsto\theta|_{T^\#}. \end{equation}
Suppose that $\theta$ is in its kernel. For any $\varphi\in S$, we then have
\[\theta(\varphi)\circ\psi\circ\theta(\varphi)^{-1} = \theta(\varphi\circ\psi\circ\varphi^{-1}) = \varphi\circ\psi\circ\varphi^{-1}\mbox{ for all }\psi\in T^\#.\]
From (\ref{centralizer}), we deduce that $\theta(\varphi) = \varphi$, and thus $\theta=\mbox{Id}_S$. This shows that the homomorphism (\ref{embedding}) is injective. Let us identify $T$ with $T^\#$ via $\sigma\mapsto\conj(\sigma)$, where $\conj(-)=\lambda(-)\rho(-)$. We then obtain an injective homomorphism
\[ \iota : \Aut(S) \longrightarrow \Aut(T)\]
from (\ref{embedding}). Since for any $\varphi\in S$ and $\sigma\in T$, we have the relation
\[ \varphi\circ\conj(\sigma)\circ\varphi^{-1} = \conj(\varphi(\sigma)) \mbox{ in }\Aut(T),\]
the stated composition is indeed the inclusion map.
\end{proof}

\section{Proof of the theorems}\label{proof sec}

In this section, we shall prove Theorems~\ref{main thm} and~\ref{main thm'}.

\subsection{Some consequences of the CFSG}
Our proof relies on the following consequences of the classification theorem of finite simple groups. 

\begin{lem}\label{ASG lem} Let $T$ be a finite non-abelian simple group.
\begin{enumerate}[(a)]
\item There is no fixed point free automorphism on $T$.
\item The outer automorphism group $\Out(T)$ of $T$ is solvable.
\item The inequality $|T|/|\Out(T)|\geq30$ holds.
\end{enumerate}
\end{lem}
\begin{proof}They are all consequences of the classification theorem of finite simple groups. See \cite[Theorems 1.46 and 1.48]{G book} for parts (a) and (b). See \cite[Lemma 2.2]{Quick} for part (c).
\end{proof}

Lemma~\ref{ASG lem} (c) in particular implies the following corollaries.

\begin{cor}\label{ASG cor}Let $T$ be a finite non-abelian simple group. Then, any finite group $S$ of order less than $30|\Aut(T)|$ cannot have subgroups $T_1$ and $T_2$, both of which are isomorphic to $T$, such that $T_1\cap T_2=1$.
\end{cor}
\begin{proof}Suppose that $S$ is a finite group having subgroups $T_1$ and $T_2$, both of which are isomorphic to $T$, such that $T_1\cap T_2=1$. Then, we have
\[ |T_1T_2| = |T_1||T_2| = |\Inn(T)||T| = |\Aut(T)||T|/|\Out(T)|.\]
Since $T_1T_2$ is a subset of $S$, from Lemma~\ref{ASG lem} (c), it follows that $S$ must have order at least $30|\Aut(T)|$.
\end{proof}

\begin{cor}\label{ASG cor'}Let $T$ be a finite non-abelian simple group. Then, the inner automorphism group $\Inn(T)$ is the only subgroup of $\Aut(T)$ isomorphic to $T$.
\end{cor}
\begin{proof}Let $R$ be any subgroup of $\Aut(T)$ isomorphic to $T$. Since $\Inn(T)\cap R$ is normal in $R$, and it cannot be trivial by Corollary~\ref{ASG cor}, it must be equal to the entire $R$. It follows that $R\subset \Inn(T)$, and we have equality because these are finite groups of the same order.
\end{proof}

\subsection{Proof of Theorem~\ref{main thm}}

Let $G$ be any finite almost simple group, say
\[ \Inn(T) \leq G \leq \Aut(T),\]
where $T$ is a finite non-abelian simple group. From Lemma~\ref{Aut(G) lem}, we see that the group $\Aut(G)$ is also almost simple, as well as that
\[ \Inn(T^\#) \leq \Aut(G)\leq \Aut(T^\#),\]
where $T^\#$ is a group isomorphic to $T$. Now, consider a regular subgroup $N$ of $\Hol(G)$. By Proposition~\ref{description1}, we may write it as
\[ N = \{\rho(\fg(\gamma))\cdot\ff(\gamma):\gamma\in\Gamma\},\mbox{ where }\begin{cases}\ff\in\Hom(\Gamma,\Aut(G))\\\fg\in Z_\ff^1(\Gamma,G)^*
\end{cases}\]
and $\Gamma$ is a group isomorphic to $N$. By Proposition~\ref{description2}, we may define
\[\fh\in\Hom(\Gamma,\Aut(G));\hspace{1em}\fh(\gamma) = \conj(\fg(\gamma))\cdot\ff(\gamma),\]
and $(\ff,\fh)$ is fixed point free since $G$ has trivial center by Lemma~\ref{center lem}. 

It is clear that
\[ \begin{cases}
N\subset\rho(G) & \mbox{if $\ff(\Gamma)$ is trivial},\\
N\subset \lambda(G) & \mbox{if $\fh(\Gamma)$ is trivial},
\end{cases}\]
which must be equalities because $N$ is regular. In what follows, assume that both $\ff(\Gamma)$ and $\fh(\Gamma)$ are non-trivial. Also, suppose for contradiction that $N$ is normal in $\Hol(G)$. Then,  by Proposition~\ref{normal prop}, both $\ff(\Gamma)$ and $\fh(\Gamma)$ are normal subgroups of $\Aut(G)$, they contain $\Inn(T^\#)$ by Lemma~\ref{socle lem}. Hence, we have
\begin{equation}\label{fh images} \Inn(T^\#) \leq \ff(\Gamma),\fh(\Gamma)\leq \Aut(G)\leq\Aut(T^\#),\end{equation}
which means that $\ff(\Gamma)$ and $\fh(\Gamma)$ are almost simple as well. 
\begin{enumerate}[(a)]
\item\textbf{Suppose that both $\ker(\ff)$ and $\ker(\fh)$ are non-trivial.} 
\\ Note that $\ker(\ff)\cap\ker(\fh)=1$ because $(\ff,\fh)$ is fixed point free. This means that $\ff$ restricts to an injective homomorphism
\[\hspace{5mm}\mbox{res}(\ff): \ker(\fh) \longrightarrow \Aut(G),\mbox{ and }\ff(\ker(\fh))\mbox{ is non-trivial}.\]
Since $\ff(\ker(\fh))$ is normal in $\ff(\Gamma)$, from Lemma~\ref{socle lem} and (\ref{fh images}), we see that $\Inn(T^\#)\leq \ff(\ker(\fh))$, and so there is a subgroup $\Delta_\fh$ of $\ker(\fh)$ isomorphic to $T$. Similarly, there is a subgroup $\Delta_\ff$ of $\ker(\ff)$ isomorphic to $T$. But
\[\hspace{5mm} \Delta_\ff\cap\Delta_\fh \subset \ker(\ff)\cap \ker(\fh)\mbox{ and so }\Delta_\ff\cap\Delta_\fh = 1.\]
This contradicts Corollary~\ref{ASG cor} because $\Gamma$ has the same order as $G$ and $G$ is contained in $\Aut(T)$ by assumption.
\vspace{1.5mm}
\item \textbf{Suppose that $\ker(\ff)$ is trivial but $\ker(\fh)$ is non-trivial.}
\\ Note that $\ff$ is injective, so $\ff$ induces an isomorphism
\begin{equation}\label{iso}\hspace{5mm}\frac{\Gamma}{\ker(\fh)}\simeq  \frac{\ff(\Gamma)}{\ff(\ker(\fh))},\mbox{ and }\ff(\ker(\fh))\mbox{ is non-trivial}.\end{equation}
On the one hand, the first quotient group in (\ref{iso}) is isomorphic to $\fh(\Gamma)$, which is insolvable by (\ref{fh images}). On the other hand, since $\ff(\ker(\fh))$ is normal in $\ff(\Gamma)$, by Lemma~\ref{socle lem} and (\ref{fh images}), we have $\Inn(T^\#)\leq \ff(\ker(\fh))$. There are natural homomorphisms
\[\hspace{5mm}\frac{\ff(\Gamma)}{\Inn(T^\#)}\xrightarrow{\text{surjective}}\frac{\ff(\Gamma)}{\ff(\ker(\fh))} \mbox{ and }\frac{\ff(\Gamma)}{\Inn(T^\#)}\xrightarrow{\text{injective}} \Out(T^\#).\]
Since $\Out(T^\#)$ is solvable by Lemma~\ref{ASG lem} (b), the second quotient group in (\ref{iso}) is also solvable, and this is a contradiction.
\vspace{1.5mm}
\item \textbf{Suppose that $\ker(\fh)$ is trivial but $\ker(\ff)$ is non-trivial.}
\\By symmetry, we obtain a contradiction as in case (b).
\vspace{1.5mm}
\item\textbf{Suppose that both $\ker(\ff)$ are $\ker(\fh)$ are trivial.}
\\  Note that $\Inn(T^\#)$ is the unique subgroup of $\Aut(G)$ isomorphic to $T$ by Corollary~\ref{ASG cor'}. Similarly, since $\Gamma\simeq\ff(\Gamma)$, from (\ref{fh images}) we see that  $\Gamma$ contains a unique subgroup $\Delta$ isomorphic to $T$. Since both $\ff$ and $\fh$ are injective, they restrict to isomorphisms
\[\hspace{5mm} \mbox{res}(\ff),\mbox{res}(\fh) : \Delta\longrightarrow\Inn(T^\#),\mbox{ and }\mbox{res}(\ff)^{-1}\circ\mbox{res}(\fh)\in\Aut(\Delta)\]
is fixed point free because the pair $(\ff,\fh)$ is fixed point free. This contradicts Lemma~\ref{ASG lem} (a).
\end{enumerate}
We have thus shown that in order for $N$ to be normal in $\Hol(G)$, either $\ff(\Gamma)$ or $\fh(\Gamma)$ must be trivial, and consequently $N$ is equal to $\lambda(G)$ or $\rho(G)$. Hence, indeed $\lambda(G)$ are $\rho(G)$ are the only elements of $\HH_1(G)$, as desired.

\subsection{Proof of Theorem~\ref{main thm'}: The first claim} Let $G$ be any finite almost simple group. Since $G$ has trivial center by Lemma~\ref{center lem}, we have
\begin{equation}\label{HH2 count'}\#\HH_2(G) = \frac{1}{|\Aut(G)|}\cdot\#\left\{(\ff,\fh):\begin{array}{c}
\ff\in\Hom(G,\Aut(G))\\ \fh\in\Hom_\ff(G,\Aut(G))^\circ\end{array}\right\}\end{equation}
as in (\ref{HH2 count}). For $(\ff,\fh)$ as above, we have $\ker(\ff)\cap\ker(\fh)=1$ since $(\ff,\fh)$ is fixed point free, and so at least one of $\ff$ and $\fh$ is injective by Lemma~\ref{socle lem}. 

\vspace{1.5mm}

Suppose now that $\Inn(G)$ is the only subgroup isomorphic to $G$ in $\Aut(G)$. If $\fh$ is injective, then $\fh(G)$ must be equal to $\Inn(G)$, and by definition (\ref{Homf}), we deduce that $\ff(G)$ has to lie in $\Inn(G)$. Then, since $\Inn(G)\simeq G$, we may identify any pair $(\ff,\fh)$ in (\ref{HH2 count'}) with $\fh$ injective as a pair $(f,h)$, where
\[ f\in\End(G)\mbox{ and } h\in\Aut(G)\mbox{ are such that }(f,h)\mbox{ is fixed point free}.\]
It follows that
\begin{align*}
&\hspace{5.5mm}\#\{(\ff,\fh):\ff\in\Hom(G,\Aut(G)),\,\fh\in\Hom_\ff(G,\Aut(G))^\circ,\,\fh\mbox{ is injective}\}\\
&= \#\{(f,h):f\in\End(G),\, h\in\Aut(G),\, (f,h)\mbox{ is fixed point free}\}\\
& = \#\{(f,h):f\in\End(G),\, h\in\Aut(G),\, h^{-1}\circ f\in\End_{\fpf}(G)\}\\
& = |\Aut(G)|\cdot \#\End_{\fpf}(G).
\end{align*}
By the symmetry between $\ff$ and $\fh$, we similarly have
\begin{align*}
&\hspace{5.5mm}\#\{(\ff,\fh):\ff\in\Hom(G,\Aut(G)),\,\fh\in\Hom_\ff(G,\Aut(G))^\circ,\,\ff\mbox{ is injective}\}\\
& = |\Aut(G)|\cdot \#\End_{\fpf}(G).
\end{align*}
We now conclude that
\begin{align*}
\#\HH_2(G) &= \frac{1}{|\Aut(G)|}\cdot2\cdot|\Aut(G)|\cdot\#\End_{\fpf}(G) \\ &= 2\cdot\#\End_{\fpf}(G)\end{align*}
and this proves the first claim in Theorem~\ref{main thm'}.


\subsection{Proof of Theorem~\ref{main thm'}: The second claim}

Observe that
\[\End_{\fpf}(G) = \bigsqcup_{H\lhd G}\{f\in\End_{\fpf}(G): \ker(f) = H\}\]
and let us begin by proving the following general statement.

\begin{lem}\label{lem}Let $G$ be any group and let $p$ be a prime. Then, for any normal subgroup $H$ of $G$ of index $p$ and any element $\sigma$ of $G$ of order $p$, we have
\[\#\{f\in\End_{\fpf}(G) : \ker(f) = H\mbox{ and }f(G)=\langle\sigma\rangle\}=\begin{cases}
p-1&\mbox{if }\sigma\in H,\\p-2 &\mbox{if }\sigma\notin H.
\end{cases}\]
\end{lem}
\begin{proof}Fix $\tau\in G$ such that $\tau H$ generates $G/H$. Then, we have exactly $p-1$ endomorphisms $f_1,\dots,f_{p-1}$ on $G$ with kernel equal to $H$ and image equal to $\langle\sigma\rangle$. Explicitly, for each $1\leq k\leq p-1$, we may define $f_k$ by setting
\begin{eqnarray*}
f_k(H) & = & \{1\},\\
f_k(\tau H) & = & \{\sigma^k\},\\
&\vdots&\\
f_k(\tau^{p-1}H) & = & \{\sigma^{k(p-1)}\}.
\end{eqnarray*}
Observe that $f_k$ is fixed point free if and only if
\[ \sigma^{ki}\in\tau^iH\mbox{ for some }i=1,\dots,p-1.\]
Since $\sigma$ and $\tau H$ have order $p$, this in turn is equivalent to $\sigma^k \in \tau H$, and
\[ \begin{cases}
\sigma^k\notin\tau H\mbox{ for all }k=1,\dots,p-1&\mbox{if }\sigma\in H,\\
\sigma^k\in\tau H\mbox{ for exactly one }k=1,\dots,p-1&\mbox{if }\sigma\notin H.
\end{cases}\]
We then see that the claim holds.
\end{proof}
Now, let $G$ be any finite almost simple group such that $\Soc(G)$ has prime index $p$ in $G$, in which case by Lemma~\ref{socle lem}, there are exactly three subgroups in $G$, namely $1$, $\Soc(G)$, and $G$. Hence, we have
\[\#\End_{\fpf}(G) = \sum_{H\in\{1,\Soc(G),G\}}\#\{f\in\End_{\fpf}(G): \ker(f) = H\}.\]
Observe that
\begin{align*}
\#\{f\in\End_{\fpf}(G): \ker(f) = 1\} & = 0,\\
\#\{f\in\End_{\fpf}(G): \ker(f) = G\} & = 1,
\end{align*}
where the former follows from Lemma~\ref{ASG lem} (a) and the latter is trivial. For the case $H=\Soc(G)$, let us first rewrite
\begin{align*}
&\hspace{5.5mm}\#\{f\in\End_{\fpf}(G): \ker(f) = \Soc(G)\}\\
&=\sum_{P\leq G,|P|=p}\#\{f\in\End_{\fpf}(G): \ker(f) = \Soc(G)\mbox{ and }f(G)= P\}\\
&=\frac{1}{p-1}\sum_{\sigma\in G,|\sigma|=p}\#\{f\in\End_{\fpf}(G): \ker(f) = \Soc(G)\mbox{ and }f(G)= \langle\sigma\rangle\}.
\end{align*}
Applying Lemma~\ref{lem} then yields
\begin{align*}
&\hspace{5.5mm}\#\{f\in\End_{\fpf}(G): \ker(f) = \Soc(G)\}\\
&=\frac{1}{p-1}\left(\sum_{\sigma\in \Soc(G),|\sigma|=p}(p-1)
+\sum_{\sigma\notin\Soc(G),|\sigma|=p}(p-2)\right)\\
& = \#\{\sigma\in\Soc(G): \sigma\mbox{ has order $p$}\}\\&\hspace{0.9cm} + (p-2)/(p-1)\cdot\#\{\sigma\in G\setminus\Soc(G) : \sigma\mbox{ has order $p$}\}.
\end{align*}
This completes the proof of the second claim in Theorem~\ref{main thm'}.

\section{Acknowledgments}

The author would like to thank the referee for helpful comments.

\end{document}